\documentclass[12pt]{amsart}

\pdfoutput=1

\usepackage{latexsym}
\usepackage[centertags]{amsmath}
\usepackage{amsfonts}
\usepackage{amssymb}
\usepackage{amsthm}
\usepackage{newlfont}
\usepackage{graphics}
\usepackage{color}

\usepackage[font=small,labelfont=bf,margin=5mm]{caption}

\RequirePackage{xparse, graphicx, caption, picins}
\DeclareDocumentCommand \addpic{O{0.4\textwidth} m g}{\parpic[r]{%
\begin{minipage}{#1}
    \includegraphics[width=\textwidth]{#2}%
    \IfNoValueTF{#3}{}{\captionof{figure}{\footnotesize #3}}
\end{minipage}
}}

\usepackage[usenames,dvipsnames]{xcolor}
\usepackage{graphicx}

\usepackage{wrapfig}

\newtheorem{theo}{Theorem}

\newtheorem{lem} [theo]{Lemma}

\newtheorem{prop}[theo]{Proposition}

\newtheorem{conj}[theo]{Conjecture}

\makeatletter \@addtoreset{equation}{section}
\@addtoreset{theo}{}\makeatother

\renewcommand{\thetheo}{\arabic{theo}}

\def\CT{\mathop{\mathop{CT}}}
\def\CT{\mathop{\mathrm{CT}}}

\newcommand\qbinom[2]{\left[#1 \atop #2  \right]_q}

\def\N{\mathbb{N}}

\newcommand{\area}{\operatorname{area}}
\newcommand{\coarea}{\operatorname{coarea}}
\newcommand{\dinv}{\operatorname{dinv}}

\newcommand{\bounce}{\operatorname{bounce}}

\def\N{\mathbb{N}}

\def\PT{\mathop{\textrm{PT}}}
\def\PTp{\mathop{\textrm{PT}'}}
\def\CT{\mathop{\textrm{CT}}}
\def\NT{\mathop{\textrm{NT}}}

\title{On Parity Unimodality of $q$-Catalan Polynomials}

\author{Guoce Xin$^*$ and Yueming Zhong }

\address{ $^1$School of Mathematical Sciences, Capital Normal University,
Beijing 100048, PR China \\
$^2$School of Mathematical Sciences, Capital Normal University,
Beijing 100048, PR China}

\email{$^*$\texttt{guoce.xin@gmail.com}\ \  \&\small   \texttt{zhongyueming107@gmail.com}}

\date{June 25, 2019} 

\begin{document}

\begin{abstract}\
A polynomial $A(q)=\sum_{i=0}^n a_iq^i$ is said to be unimodal if $a_0\le a_1\le \cdots \le a_k\ge a_{k+1} \ge \cdots \ge a_n$.
We investigate the unimodality of rational $q$-Catalan polynomials, which is defined to be $C_{m,n}(q)= \frac{1}{[n+m]} \qbinom{m+n}{n}$
for a coprime pair of positive integers $(m,n)$. We conjecture that
they are unimodal with respect to parity, or equivalently, $(1+q)C_{m+n}(q)$ is unimodal. By using generating functions and the constant term method,
we verify our conjecture for $m\le 5$ in a straightforward way.
\end{abstract}

\maketitle
{\small Mathematics Subject Classifications: 05A15, 05A20, 05E05}

{\small \textbf{Keywords}: rational Dyck paths; rational $q$-Catalan polynomials; unimodal sequences.}

\section{Introduction}\label{sec-1-intro}
We will consider the unimodality of some symmetric polynomials.
A sequence $a_0,a_1,\dots, a_n$ is said to be symmetric if $a_i=a_{n-i}$ for all $i$.
It is said to be unimodal if there is a $k$ such that $a_0\le a_1\le \cdots \le a_k\ge a_{k+1} \ge \cdots \ge a_n$.
It is said to be unimodal with respect to parity if $a_0,a_2,\dots $ and $a_1,a_3,\dots $ are both unimodal.
A polynomial $P(q)=a_0+a_1q+\cdots +a_n q^n$ of degree $n$ is said to be symmetric (resp., unimodal) if its coefficient sequence $a_0,a_1,\dots, a_n$ is symmetric (resp., unimodal).

Stanley gave a nice survey \cite{ref-8} on various methods for showing that a sequence is unimodal (or log-concave which we will not discuss here).
A classical example is the following.
\begin{theo}\label{G-unimodal}
  The Gaussian polynomial $G_{m,n}(q)$ is symmetric and unimodal, where
  $$ G_{m,n}(q)= \qbinom{m+n}{m} = \frac{[m+n]!}{[m]![n]!},$$
with the classical $q$-notation $[n]=\frac{1-q^n}{1-q},\ [n]!=[n][n-1]\cdots [1]$.
\end{theo}
This important result has many proofs. See \cite{ref-2,ref-7,ref-1,ref-5,ref-4,ref-8,ref-3,ref-6}. It is an outstanding open question to find an explicit order matching proof for the unimodality.
Though the Gaussian polynomials have been extensively studied, the closely related $q$-Catalan polynomials are less studied.

The $q$-Catalan polynomials (or numbers) $C_n(q)$ we are discussing here are defined by
$$C_n(q)= \frac{1}{[n+1]} \qbinom{2n}{n} = \frac{[2n]!}{[n+1]![n]!}.$$
It starts with $C_0(q)=C_1(q)=1,\ C_2(q)=1+q^2,\ C_3(q)=1+q^2+q^3+q^4+q^6,$
$C_4(q)=1+q^2+q^3+2q^4+q^5+2q^6+q^7+2q^8+q^9+q^{10}+q^{12}.$
Clearly, one sees that $C_n(q)$ is symmetric, but not unimodal. However, we have the following conjecture.

\begin{conj}\label{qCn-Conjecture}
  The $q$-Catalan polynomials $C_n(q)$ are unimodal with respect to parity.
\end{conj}

We find that Conjecture \ref{qCn-Conjecture} can be extended for rational $q$-Catalan polynomials.
For a pair $(m,n)$ of positive integers, define
$$C_{m,n}(q)= \frac{1}{[n+m]} \qbinom{m+n}{n}=\frac{1}{[m]} \qbinom{m+n-1}{n}.$$
When $m$ and $n$ are coprime to each other, i.e., $\gcd(m,n)=1$, $C_{m,n}(q)$ is known to be in $\mathbb{N}[q]$ (a polynomial with nonnegative coefficients), and is called
the $(m,n)$ rational $q$-Catalan polynomials (or $q$-Catalan numbers). See, e.g., \cite{Haiman}.
\begin{conj}\label{qCmn-Conjecture}
  For a coprime pair of positive integers $(m,n)$, the $(m,n)$-rational $q$-Catalan polynomials $C_{m,n}(q)$ are unimodal with respect to parity.
\end{conj}

When $\gcd(m,n)>1$, $C_{m,n}(q)$ is usually not a polynomial, while
it has been shown that $\bar{C}_{m,n}(q)=[\gcd(m,n)] C_{m,n}(q)$ is in $\mathbb{N}[q]$. See, e.g., \cite{mnCatlanAndrews,ref-14,mnCatlanKratt}.
\begin{conj}\label{qCmn-2Conjecture}
  For a pair of positive integers $(m,n)$, the polynomial $\bar{C}_{m,n}(q)$ is unimodal with respect to parity.
\end{conj}

Conjecture \ref{qCmn-2Conjecture} includes Conjecture \ref{qCmn-Conjecture} as a special case, since $\bar{C}_{m,n}(q)$ reduces to $C_{m,n}(q)$ when $\gcd(m,n)=1$. We state the latter separately because
the combinatorial meaning of $C_{m,n}(q)$ is much more elegant as we will explain later in Section \ref{sec-5-com}. Conjecture \ref{qCmn-Conjecture} includes Conjecture \ref{qCn-Conjecture} as a special case, since
$C_{n+1,n}(q)=C_n(q)$. We state the latter separately because $C_n(q)$ has a different combinatorial interpretation. See Section \ref{sec-5-com}.

We have verified Conjecture \ref{qCmn-2Conjecture} for $m,n\le 180$ by Maple. Observe that $\bar{C}_{m,km}(q)=[m]C_{m,km}= \qbinom{km+m-1}{km}$, which by Theorem \ref{G-unimodal} is indeed unimodal.
We will prove this conjecture for $m\le 5$. Our method is to compute the corresponding
generating functions by means of the constant term method. It turns out that for $m\le 5$, the positivity is transparent in view of their generating functions.

This paper is organized as follows. Section \ref{sec-1-intro} is this introduction. Section \ref{sec-2-pre} introduces the basic concepts and idea for our proof.
The unimodality conjectures are translated by using their generating functions. Section \ref{sec-3-dir} tries direct computation,
which already becomes complicated for $m=4$.
Section \ref{sec-4-gen} uses the constant term method to compute the corresponding generating functions, whose positivity is transparent and hence proves Conjecture \ref{qCmn-Conjecture} for $m\le 5$.
In Section \ref{sec-5-com}, we discuss possible representation approach for settling these conjectures. We also introduce the combinatorial interpretations of $q$-Catalan polynomials.

\def\N{{\mathcal{N}}}
\section{Preliminary \label{sec-2-pre}}
A Laurent polynomial $L(q)$ is said to be symmetric if $L(q^{-1})=L(q)$, and is said to be anti-symmetric if $L(q^{-1})=-L(q)$.

Define the normalization of a polynomial $P(q)$ of degree $n$ by
$$ \mathcal{N} P(q) := \mathcal{N}( P(q))= P(q^2)q^{-n}, \qquad \N( 0)=0.$$
Then the symmetry of $P(q)$ (i.e., $P(q)=q^n P(q^{-1})$) is transformed to the more natural symmetry of the Laurent polynomial $\N {P}(q)$.
The following properties are easy to verify and will be used without mentioning:

\begin{enumerate}
  \item For polynomials $P(q)$ and $Q(q)$, we have $  \mathcal{N} (P(q) Q(q)) =  \mathcal{N}( P(q))  \mathcal{N}(Q(q))$.

  \item The product of two symmetric Laurent polynomials is still symmetric.

  \item If $L_1(q^{-1})=L_1(q)$ and $L_2(q^{-1})=-L_2(q)$, then $L_1(q)L_2(q) \big|_{q=q^{-1}} = -  L_1(q)L_2(q).$
\end{enumerate}


We will also use the following linear operators on Laurent polynomials in $\mathbb{Q}[q,q^{-1}]$.
\begin{align*}
  \PT_q \sum_{i} a_i q^i &= \sum_{i>0} a_i q^i,  \qquad   \text{ (extracting the positive exponent terms) } \\
  \CT_q \sum_{i} a_i q^i &= a_0,  \qquad   \text{ (extracting the constant term) } \\
  \NT_q \sum_{i} a_i q^i &= \sum_{i<0} a_i q^i,  \qquad   \text{ (extracting the negative exponent terms) } \\
\end{align*}
These operators clearly extend to $\mathbb{Q}[q,q^{-1}][[x]]$, the ring of power series in $x$ with coefficients Laurent polynomials in $q$. Indeed, they act
coefficient wise in $x$.

The following lemma is transparent.
\begin{lem}\label{lem-1}
Suppose $P(q)$ is a symmetric polynomial of degree $n$.
Then£¬
\begin{enumerate}
  \item $P(q)$ is unimodal if and only if $ \PT_q \N((q-1)P(q))
  \in \mathbb{N}[q]$.
  \item $P(q)$ is unimodal with respect to parity if and only if
  $$ \PT_q  \N (q^2-1)P(q) = \PT_q  (q^2-q^{-2}) \N P(q) \in \mathbb{N}[q].$$
\end{enumerate}
\end{lem}

Thus Conjecture \ref{qCmn-Conjecture} can be restated as follows.
\addtocounter{theo}{-1}
\renewcommand\thetheo{\ref{qCmn-Conjecture}$^a$}
\begin{conj}\label{qCmn-Conjecturea}
  For a coprime pair of positive integers $(m,n)$, the polynomial
\begin{align}
\PT_q \N (q^2-1) &C_{m,n}(q)=\PT_q (q^2-q^{-2}) C_{m,n}(q^2)q^{-(m-1)(n-1)} \nonumber\\
&=\PT_q(q^2-q^{-2})\frac{(1-q^{2m+2})(1-q^{2m+4})\cdots(1-q^{2m+2n-2})}{(1-q^4)(1-q^6)\cdots(1-q^{2m})}q^{-(m-1)(n-1)}
\end{align}
  has nonnegative integer coefficients.
\end{conj}
\renewcommand\thetheo{\arabic{theo}}
Our point is that it is usually easier to consider generating functions. Let
\begin{align}
  F_m(x,q):= \sum_{n\ge 0} (q^2-q^{-2})\N C_{m,n}(q) x^n.
\end{align}
Note that the coefficients in $x$ are not always Laurent polynomials in $q$. We take $F_m(x,q)$ as an element in $\mathbb{Q}((q))[[x]]$, the ring of
power series in $x$ with coefficients Laurent series in $q$.

For integers $m> r\ge 0$, let $X_{m,r}$ be the linear operator acting on $\mathbb{Q}(q)[[x]]$ by
 \begin{align}
   X_{m,r} \sum_{n\ge 0} a_n(q) x^n &= \sum_{k\ge 0} a_{km+r}(q) x^k.  \quad   \text{ (extracting the terms with special exponents) }
 \end{align}

When $\gcd(m,r)=1$, $X_{m,r} F_m(x,q) \in \mathbb{Q}[q,q^{-1}][[x]]$.

Then Conjecture \ref{qCmn-Conjecturea} is transformed as follows.
\addtocounter{theo}{-1}
\renewcommand\thetheo{\ref{qCmn-Conjecture}$^b$}
\begin{conj}\label{qCmn-Conjectureb}
Let $m>r$ be positive integers with $\gcd(m,r)=1$. Then
  $$ \PT_q  X_{m,r} F_m(x,q) = X_{m,r} \PT_q F_m(x,q)$$
is a power series in $x$ with coefficients in $\mathbb{N}[q]$.
\end{conj}
\renewcommand\thetheo{\arabic{theo}}

The case $\gcd(m,r)=d$ is a little complicated. We need to consider the generating function
\begin{align*}
  \sum_{k\ge 0} \N ((q^2-1)[d] C_{m,km+r}(q)) x^k &= \sum_{k\ge 0} \frac{q^d-q^{-d}}{q-q^{-1}} (q^2-q^{-2}) \N ( C_{m,km+r}(q)) x^k \\
                                           &=  \frac{q^d-q^{-d}}{q-q^{-1}} X_{m,r} \sum_{n\ge 0} (q^2-q^{-2}) \N ( C_{m,n}(q)) x^n
                                           \\
                                           &=  \frac{q^d-q^{-d}}{q-q^{-1}} X_{m,r} F_m(x,q).
\end{align*}
Thus  Conjecture \ref{qCmn-2Conjecture} can be transformed as follows.
\addtocounter{theo}{-1}
\renewcommand\thetheo{\ref{qCmn-2Conjecture}$^b$}
\begin{conj}\label{qCmn-2Conjectureb}
Let $m>r$ be nonnegative integers with $\gcd(m,r)=d$. Then
  $$ \PT_q  X_{m,r} \frac{q^{d}-q^{-d}}{q-q^{-1}}  F_m(x,q) = X_{m,r} \PT_q \frac{q^{d}-q^{-d}}{q-q^{-1}} F_m(x,q)$$
is a power series in $x$ with coefficients in $\mathbb{N}[q]$.
\end{conj}
\renewcommand\thetheo{\arabic{theo}}
We remark that $\frac{q^{d}-q^{-d}}{q-q^{-1}} F_m(x,q)$ is a power series in $x$ with coefficient Laurent series in $q$, so we need to extend
the $\PT_q$ operator. See Section \ref{sec-4-gen}.

To our surprise, $F_m(x,q)$ has a product formula as follows.
\begin{prop}\label{Fxq}
For any positive integer $m$, we have
\begin{align}
F_m(x,q)  = (q^2-q^{ - 2})\frac{(q-q^{ - 1})}{(q^m-q^{ - m})}\prod_{i = 0}^{m - 1}\frac{1}{ (1-xq^{1 - m}\cdot {q^{2i}})}
\end{align}
\end{prop}
\begin{proof}
The proposition is indeed a consequence of the following well-known identity. See, e.g., \cite[pp. 272]{ref-12}.
\begin{align}
\frac{1}{(1-x)(1-xq)(1-x{q^2}) \cdots (1-x{q^m})}= \sum_{n>0} \qbinom{m+n}{m}x^n
\end{align}

We have
\begin{align*}
&\sum_{n>0} q^{\frac{-(n-1)(m-1)-2}{2}}(q^2-1)C_{m,n}(q)x^n\\
&=\sum_{n>0} q^{\frac{-(n-1)(m-1)-2}{2}}(q^2-1)\frac{1-q}{1-q^m}\qbinom{m+n-1}{m-1}x^n\\
&=q^{\frac{m-3}{2}}(q^2-1)\frac{1-q}{1-q^m}\sum_{n>0}\qbinom{m+n-1}{m-1}(q^{\frac{-(m-1)}{2}}x)^n\\
&=q^{\frac{m-3}{2}}(q^2-1)\frac{1-q}{1-q^m}\prod_{i = 0}^{m - 1} \frac{1}{(1-xq^{\frac{-(m-1)}{2}}q^i)}.
\end{align*}
We can get
\begin{align*}
F_m(x,q)&=\sum_{n>0}(q^2-q^{-2}) \N {C}_{m,n}(q)x^n\\
&=\sum_{n>0}(q^2-q^{-2}) C_{m,n}(q^2)q^{-(m-1)(n-1)}x^n \\
&=\sum_{n>0} q^{-(n-1)(m-1)-2}(q^4-1)C_{m,n}(q^2)x^n\\
&=q^{m-3}(q^4-1)\frac{1-q^2}{1-q^{2m}}\prod_{i = 0}^{m - 1} \frac{1}{ (1-xq^{1 - m}\cdot {q^{2i}})}\\
&=(q^2-q^{ - 2})\frac{(q-q^{ - 1})}{(q^m-q^{ - m})}\prod_{i = 0}^{m - 1}\frac{1}{ (1-xq^{1 - m}\cdot {q^{2i}})}.
\end{align*}
\end{proof}

\section{Direct computation}\label{sec-3-dir}
Conjecture \ref{qCmn-Conjecture} can be verified directly for $m=3$, but already becomes complicated for $m=4$.
\subsection{The case $m=3$}
In this case, we have the following explicit expansion.
\begin{prop}
We have
\begin{align}
  (q^2-1) C_{3,n}(q) &= \left\{\begin{array}{ll}
                           q^{3k+1}(- \sum_{i = 0}^k q^{-(3i+1)}+\sum_{i = 0}^k q^{3i+1}), & \text{ if } n=3k+1; \\
                           q^{3k+2}(- \sum_{i = 0}^k q^{-(3i+2)}+\sum_{i = 0}^k q^{3i+2}), & \text{ if } n=3k+2.
                          \end{array}   \right.
\end{align}
Consequently, Conjecture \ref{qCmn-Conjecture} holds true for $m=3$.
\end{prop}
\begin{proof}
By direct computation, we have
\begin{align*}
  (q^2-1) C_{3,n}(q) &= \frac{(q^2-1)(1-q)}{1-q^{n+3}} \frac{(1-q^{n+1})(1-q^{n+2})(1-q^{n+3})}{(1-q)(1-q^2)(1-q^3)} \\
  &= - \frac{(1-q^{n+1})(1-q^{n+2})}{(1-q^3)}
\end{align*}
When $n=3k+1$, we have
\begin{align*}
  (q^2-1) C_{3,3k+1}(q) &= - \frac{(1-q^{3k+2})(1-q^{3k+3})}{(1-q^3)} \\
   &= (q^{3k+2}-1) (1+q^3+q^6+\cdots +q^{3k})\\
   &={q^{3k + 1}}\left( { - \sum\limits_{i = 0}^k {{q^{ - \left( {3i + 1} \right)}} + \sum\limits_{i = 0}^k {{q^{3i + 1}}} } } \right).
\end{align*}
This proves Conjecture \ref{qCmn-Conjecture} for $(m,n)=(3,3k+1)$.

When $n=3k+2$, we have
\begin{align*}
  (q^2-1) C_{3,3k+2}(q) &= - \frac{(1-q^{3k+3})(1-q^{3k+4})}{(1-q^3)} \\
     &= (q^{3k+4}-1) (1+q^3+q^6+\cdots +q^{3k})\\
     &={q^{3k + 2}}\left( { - \sum\limits_{i = 0}^k {{q^{ - \left( {3i + 2} \right)}} + \sum\limits_{i = 0}^k {{q^{3i + 2}}} } } \right).
\end{align*}
This proves Conjecture \ref{qCmn-Conjecture} for $(m,n)=(3,3k+2)$.
\end{proof}

\subsection{The case $m=4$}
This case is already complicated. We have
\begin{align*}
  (q^2-1) C_{4,n}(q) &= - \frac{(1-q^{n+1})(1-q^{n+2})(1-q^{n+3})}{(1-q^3)(1-q^4)}.
\end{align*}
We can have explicit polynomial representation, but that will not help to prove our conjecture. For instance, if $n=12k+1$, then
\begin{align*}
  (q^2-1) C_{4,12k+1}(q) &= - \frac{(1-q^{12k+2})(1-q^{12k+3})(1-q^{12k+4})}{(1-q^3)(1-q^4)} \\
                         &=(q^{12k+2} -1)\cdot  \sum_{i=0}^{4k} q^{3i} \cdot \sum_{j=0}^{3k} q^{4j}.
\end{align*}
Now it is unclear why its coefficients in $q^r$ is negative for $r\le \frac{3n-1}{2}$.

\section{The generating function method}\label{sec-4-gen}

\subsection{Basic idea}
We illustrate the idea by redoing the case $m=3, \ n=3k+1$.
Consider the generating function
\begin{align*}
 \sum_{k\ge 0}(q-q^{-1})q^{-3k} C_{3,3k+1}(q) x^k &= \sum_{k\ge 0} -q^{-3k-1} \frac{(1-q^{3k+2})(1-q^{3k+3})}{(1-q^3)} x^k \\
 &=\frac{1}{1-q^3}\sum_{k\ge 0} (-q^{-3k-1}+q+q^2-q^{3k+4})x^k\\
 &=\frac{1}{1-q^3}\left(-\frac{q^{-1}}{1-q^{-3}x} +\frac{q^2+q}{1-x}-\frac{q^4}{1-q^{3}x} \right)\\
 &=\frac{q \left({1-q}\right)\left({1+q}\right)\left({x+q+{q}^{2}x}\right)} {\left({1-x}\right)\left({x-{q}^{3}}\right)\left({1-{q}^{3}x}\right)}.
\end{align*}
By taking partial fraction decompositions in $q$, we obtain:%
\begin{align*}
\frac{q \left({1-q}\right)\left({1+q}\right)\left({x+q+{q}^{2}x}\right)} {\left({1-x}\right)\left({x-{q}^{3}}\right)\left({1-{q}^{3}x}\right)}
&=\frac{q^2}{(q^3-x)(x-1)}+\frac{q}{(q^3x-1)(x-1)}.
\end{align*}
When expanding as a power series in $x$, the first term has only negative powers in $q$ and the second term has only positive powers in $q$:
\begin{align*}
 \frac{q^2}{(q^3-x)(x-1)}&=\frac{-q^{-1}}{(1-q^{-3}x)(1-x)}=-q^{-1}\Big(\sum_{i\ge 0} x^i\Big)\Big(\sum_{i\ge 0}(q^{-3}x)^i \Big),\\
 \frac{q}{(q^3x-1)(x-1)} &=q\Big(\sum_{i\ge 0} x^i\Big)\Big(\sum_{i\ge 0}(q^{3}x)^i\Big).
\end{align*}
Thus by extracting positive powers in $q$, we obtain
\begin{align*}
\sum_{k\ge 0} \PT_{q} (q-q^{-1})q^{-3k} C_{3,3k+1}(q) x^k
&=\frac{q}{(q^3x-1)(x-1)} \in \mathbb{N}[q][[x]].
\end{align*}
Conjecture \ref{qCmn-Conjectureb} thus holds for the case $(m,r)=(3,1)$.

The case $(m,r)=(3,2)$ can be done similarly.

Extracting positive powers in $q$ of a general class of rational series can be done systematically by the constant term method.
The resulting rational function turns out to be trivially positive for $m\le 5$.

\subsection{The constant term method}
Constant term extraction or residue computation has a long history. See, e.g., \cite{XinThesis} for further references. The fundamental problem
 we are concerned here is to compute the constant term of in a set of variables of a formal series in the field of iterated Laurent series $K=\mathbb{Q}((x_n))\cdots ((x_1))$,
which is called the working field. The reader is referred to \cite{ref-14} for the original development of the field of iterated Laurent series.
Here we only recall that $K$ defines a total ordering $0<x_1<x_2<\cdots <x_n<1$ on the variables (more precisely, a total group order on its monomials),
which can be formally treated as $0<\!\!<x_1<\!\!<x_2<\!\!<\cdots <\!\!< x_n<\!\!<1$.
Every rational function has a unique series expansion in $K$. We will focus on the class of Elliott-rational functions, which are rational functions
whose denominators are the product of binomials. It is known
that the constant term of an Elliott-rational function is still an Elliott-rational function. Efficient algorithms have been developed to evaluate this type of constant terms,
such as the Omega Mathematica package \cite{OmegaUp,Omega}, Ell Maple package \cite{Xin-fast} developed from Algebraic Combinatorics.
See \cite{Xin-CTEuclid} for further references on algorithmic development from Computational Geometry and Algebraic Combinatorics.

We will use the first author's (updated) Ell2 Maple package. We use a list $xvar=[x_1,x_2,\dots, x_n]$ to specify the working field $\mathbb{Q}((x_n))\cdots ((x_1))$.
Let $var=[x_{i_1},\dots, x_{i_s}]$ be a list of variables to
be eliminated, then the constant term of an Elliott-rational function $F(x_1,\dots, x_n)$
$$ \CT_{x_{i_1},\dots, x_{i_s}} F(x_1,\dots, x_n) $$
can be evaluated by the command \texttt{$E\_OeqW(F,xvar,var)$} after loading the Ell2 package. 

In what follows, we always specify the working field $K$ by letting $0<x<q<\lambda<1$. This $K$ includes all the rings, such as $\mathbb{Q}((q))[[x]]$,  appear below as a subring.
Firstly, we shall explain how to realize the $\PT_q$ and $X_{m,r}$ operators by the constant term operator.

For anti-symmetric Laurent polynomials $L(q)$, we have $\PT_q L(q) = -\NT_q  L(q) \Big|_{q=q^{-1}}$.
For anti-symmetric $F(x,q)\in \mathbb{Q}((q))[[x]]$, $\PT_q F(x,q)$ is not in $\mathbb{Q}[q][[x]]$, but $\NT_q F(x,q)$ belongs to $\mathbb{Q}[q^{-1}][[x]]$. It is convenient for us to use
$$\PTp_q F(x,q)= -\NT_q  F(x,q) \Big|_{q=q^{-1}}$$ to
replace $\PT_q F(x,q)$, since they agree when $F(x,q)\in Q[q,q^{-1}][[x]]$. We have
\begin{align}
\PTp_q F(x,q) = -\NT_q  F(x,q) \Big|_{q=q^{-1}} & = -\CT_\lambda  \frac{\lambda q}{1-\lambda q} F(x,\lambda),    \\
X_{m,r} F(x,q) &= \CT_\lambda  \frac{\lambda^{-r}}{1-x \lambda^{-m}} F(\lambda, q).
\end{align}
The above identities are easily verified for $F(x,q)=q^ix^j$ and then extended by linearity for arbitrary $F(x,q)$.

Let us redo the $m=3$ case for the sake of clarity. The cases $n=3k+1$ and $n=3k+2$ can be done simultaneously.
By starting with the formula
$$ F_3(x,q)=\frac{q^{2}\left({1-{q}^{2}}\right)\left({1-{q}^{4}}\right)} {\left({1-x}\right)\left({x-{q}^{2}}\right)\left({1-{q}^{2}x}\right)\left({1-{q}^{6}}\right)},$$
we compute
\begin{align*}
  G_3(x,q)= \PTp_q  F_3(x,q) &=\frac{q^{2}x } {\left({1-{x}^{3}}\right)\left({1-{q}^{2}x}\right)}.
\end{align*}
This clearly belongs to $\mathbb{N}[q][[x]]$, and hence reprove Conjecture \ref{qCmn-Conjectureb} for $(m,r)=(3,1)$ and $(m,r)=(3,2)$. Indeed, a further step gives
$$ X_{3,1} G_3(x,q) =\frac{q^{2}} {\left({1-x}\right)\left({1-{q}^{6}x}\right)} \in \mathbb{N}[q][[x]], $$
$$ X_{3,2} G_3(x,q)=\frac{q^{4}} {\left({1-x}\right)\left({1-{q}^{6}x}\right)} \in \mathbb{N}[q][[x]]. $$
The case $n=3k$ is a little different. We need to compute
$$  H_3^0(x,q) = \PTp_q \frac{q^3-q^{-3}}{q-q^{-1}} F_3(x,q) =\frac{q^{2}} {\left({1-x}\right)\left({1-{q}^{2}x}\right)},$$
which clearly belongs to $\mathbb{N}[q][[x]]$. This implies the positivity of $X_{3,0} H_3^0(x,q)$ and hence reproves Conjecture \ref{qCmn-2Conjectureb} for $(m,r)=(3,0)$. Indeed, we have
$$ X_{3,0} H_3^0(x,q) =\frac{q^{2}\left({1+{q}^{2}x+{q}^{4}x}\right)} {\left({1-x}\right)\left({1-{q}^{6}x}\right)}.  $$

\subsection{The case $m=4$}
We shall establish the following result.
\begin{prop}
  Conjecture \ref{qCmn-2Conjectureb} holds true for $m=4$.
\end{prop}
\begin{proof}
We start with the formula
$$ F_4(x,q)=-\frac{q^{5}\left({1-{q}^{2}}\right)\left({1-{q}^{4}}\right)} {\left({x-q}\right)\left({1-qx}\right)\left({x-{q}^{3}}\right)\left({1-{q}^{3}x}\right)\left({1-{q}^{8}}\right)}.$$
By constant term extraction, we obtain
\begin{align*}
  G_4(x,q)= \PTp_q  F_4(x,q)  &= -\frac{x q \left({-q+{x}^{3}-q{x}^{2}+{q}^{2}x-q{x}^{4}}\right)} {\left({1-qx}\right)\left({1-x{q}^{3}}\right)\left({1-{x}^{8}}\right)},
\end{align*}
which do not show the positivity directly. By applying $X_{4,r}$ for $r=1,3$, we obtain
\begin{align*}
  X_{4,1} G_4(x,q) & = \frac{q^{2}\left({1+{q}^{6}x+{q}^{6}{x}^{2}+{q}^{12}{x}^{2}}\right)} {\left({1-{x}^{2}}\right)\left({1-{q}^{4}x}\right)\left({1-{q}^{12}x}\right)},\\
  X_{4,3} G_4(x,q) &=\frac{q^{2}\left({1+{q}^{6}+{q}^{6}x+{q}^{12}{x}^{2}}\right)} {\left({1-{x}^{2}}\right)\left({1-{q}^{4}x}\right)\left({1-{q}^{12}x}\right)}.
\end{align*}
This proves Conjecture \ref{qCmn-Conjectureb} for $(m,r)=(4,1), \ (4,3)$.

For the case $(m,r)=(4,2)$, we need to compute
\begin{align*}
  H_4^2(x,q) = \PTp_q \frac{q^2-q^{-2}}{q-q^{-1}}  F_4(x,q) &= \frac{x q \left({1-{x}^{4}}\right)\left({1-qx+{q}^{2}-{q}^{3}x+{q}^{3}{x}^{3}}\right)} {\left({1-{x}^{2}}\right)\left({1-qx}\right)\left({1-{q}^{3}x}\right)\left({1-{x}^{8}}\right)},
\end{align*}
$$X_{4,2} H_4^2(x,q)=\frac{q^{4}\left({1+{q}^{2}}\right)\left({1+x{q}^{6}}\right)} {\left({1-x}\right)\left({1-{q}^{4}x}\right)\left({1-{q}^{12}x}\right)},$$
which is clearly in $\mathbb{N}[q][[x]]$. %

For the case $(m,r)=(4,0)$, we need to compute
\begin{align*}
  H_4^0(x,q) = \PTp_q \frac{q^4-q^{-4}}{q-q^{-1}}  F_4(x,q) &=  \frac{q^2} {\left({1-{x}^{2}}\right)\left({1-qx}\right)\left({1-{q}^{3}x}\right)},
\end{align*}
$$X_{4,0} H_4^0(x,q)=\frac{q^{2}\left({1+{q}^{2}x+{q}^{4}x+2\,x{q}^{6}+{q}^{8}x+{q}^{10}x+{q}^{12}{x}^{2}}\right)} {\left({1-x}\right)\left({1-{q}^{4}x}\right)\left({1-{q}^{12}x}\right)},$$
which is clearly in $\mathbb{N}[q][[x]]$. %
%
\end{proof}

\subsection{The case $m=5$}
\begin{prop}
  Conjecture \ref{qCmn-2Conjectureb} holds true for $m=5$.
\end{prop}
\begin{proof}
We start with the formula
$$ F_5(x,q)=-\frac{q^{8}\left({1-{q}^{2}}\right)\left({1-{q}^{4}}\right)} {\left({1-x}\right)\left({x-{q}^{2}}\right)\left({1-{q}^{2}x}\right)\left({x-{q}^{4}}\right)\left({1-{q}^{4}x}\right)\left({1-{q}^{10}}\right)}.$$
By the constant term method, we obtain
\begin{align*}
  G_5(x,q)= \PTp_q  F_5(x,q) &= \frac{x q^{2}\left({1-x}\right)\left({1+x-{x}^{3}-{q}^{2}x+{q}^{2}{x}^{3}+{q}^{2}{x}^{4}}\right)} {\left({1-{x}^{2}}\right)\left({1-{x}^{3}}\right)\left({1-{q}^{2}x}\right)\left({1-{x}^{5}}\right)\left({1-{q}^{4}x}\right)},
\end{align*}
which do not show the positivity directly.
By applying $X_{4,r}$ for $r=1,2,3,4$, we obtain
\begin{align*}
  X_{5,r} G_5(x,q)= \frac{q^2P_{5,r}(x,q)}{{\left({1-{x}^{2}}\right)\left({1-{x}^{3}}\right)\left({1-{q}^{10}x}\right)\left({1-{q}^{20}x}\right)}},
\end{align*}
where
\begin{align*}
P_{5,1}(x,q)&=1 + {q^2}x + {q^2}{x^2} + {q^4}x + {q^4}{x^2} + {q^4}{x^3} + {q^6}{x^2} + 2{\mkern 1mu} {q^6}{x^3} + {q^8}x + {q^6}{x^4}\\
&+ 2{\mkern 1mu} {q^8}{x^2} + {q^8}{x^3} + {q^8}{x^4} + 2{\mkern 1mu} {q^{10}}{x^2} + 2{\mkern 1mu} {q^{10}}{x^3} + {q^{12}}x + {q^{10}}{x^4} + {q^{12}}{x^2}\\
&+ 2{\mkern 1mu} {q^{12}}{x^3} + {q^{14}}x + {q^{12}}{x^4} + {q^{14}}{x^2} + {q^{14}}{x^3} + {q^{14}}{x^4} + 2{\mkern 1mu} {q^{16}}{x^2} + {q^{14}}{x^5}\\
&+ 2{\mkern 1mu} {q^{16}}{x^3} + {q^{16}}{x^4} + {q^{18}}{x^2} + 2{\mkern 1mu} {q^{18}}{x^3} + 2{\mkern 1mu} {q^{18}}{x^4} + {q^{20}}{x^2} + {q^{20}}{x^3} + {q^{20}}{x^4}\\
&+ {q^{20}}{x^5} + {q^{22}}{x^3} + {q^{22}}{x^4} + {q^{22}}{x^5} + {q^{24}}{x^3} + {q^{24}}{x^4} + {q^{26}}{x^2},
\end{align*}
\begin{align*}
  P_{5,2}(x,q)&=x + {q^2}{x^2} + {q^4} + {q^2}{x^3} + {q^4}x + {q^4}{x^2} + {q^6}x + {q^6}{x^2} + {q^6}{x^3} + {q^8}x\\
&+ {q^6}{x^4} + 2{\mkern 1mu} {q^8}{x^2} + 2{\mkern 1mu} {q^8}{x^3} + {q^{10}}x + {q^{10}}{x^2} + 2{\mkern 1mu} {q^{10}}{x^3} + {q^{12}}x + {q^{10}}{x^4}\\
&+ 2{\mkern 1mu} {q^{12}}{x^2} + {q^{12}}{x^3} + {q^{12}}{x^4} + 2{\mkern 1mu} {q^{14}}{x^2} + 2{\mkern 1mu} {q^{14}}{x^3} + {q^{16}}x + {q^{14}}{x^4} + {q^{16}}{x^2}\\
&+ 2{\mkern 1mu} {q^{16}}{x^3} + {q^{18}}x + {q^{16}}{x^4} + {q^{18}}{x^2} + {q^{18}}{x^3} + {q^{18}}{x^4} + {q^{20}}{x^2} + {q^{18}}{x^5} + 2{\mkern 1mu} {q^{20}}{x^3}\\
&+ {q^{20}}{x^4} + {q^{22}}{x^2} + {q^{22}}{x^3} + {q^{22}}{x^4} + {q^{24}}{x^4} + {q^{24}}{x^5} + {q^{26}}{x^3},
\end{align*}
\begin{align*}
  P_{5,3}(x,q)&={x^2} + {q^2} + {q^2}x + {q^4}x + {q^4}{x^2} + {q^4}{x^3} + {q^6}x + 2{\mkern 1mu} {q^6}{x^2} + {q^8} + {q^6}{x^3} + {q^8}x\\
 &+ {q^8}{x^2} + {q^8}{x^3} + {q^{10}}x + {q^8}{x^4} + 2{\mkern 1mu} {q^{10}}{x^2} + {q^{10}}{x^3} + {q^{12}}x + {q^{10}}{x^4} + 2{\mkern 1mu} {q^{12}}{x^2}\\
 &+ 2{\mkern 1mu} {q^{12}}{x^3} + {q^{14}}x + {q^{14}}{x^2} + 2{\mkern 1mu} {q^{14}}{x^3} + {q^{16}}x + {q^{14}}{x^4} + 2{\mkern 1mu} {q^{16}}{x^2} + {q^{16}}{x^3} + {q^{16}}{x^4}\\
 &+ 2{\mkern 1mu} {q^{18}}{x^2} + 2{\mkern 1mu} {q^{18}}{x^3} + {q^{20}}x + {q^{18}}{x^4} + {q^{20}}{x^2} + {q^{20}}{x^3} + {q^{20}}{x^4} + {q^{22}}{x^3} + {q^{22}}{x^4}\\
 &+ {q^{24}}{x^2} + {q^{22}}{x^5} + {q^{24}}{x^3} + {q^{26}}{x^4},
 \end{align*}
 \begin{align*}
  P_{5,4}(x,q)&={x^3} + {q^2}x + {q^2}{x^2} + {q^4} + {q^4}x + {q^4}{x^2} + {q^6} + {q^6}x + {q^6}{x^2} + {q^6}{x^3} + 2{\mkern 1mu} {q^8}x + 2{\mkern 1mu} {q^8}{x^2}\\
 &+ {q^8}{x^3} + {q^{10}}x + 2{\mkern 1mu} {q^{10}}{x^2} + {q^{12}} + 2{\mkern 1mu} {q^{10}}{x^3} + {q^{12}}x + {q^{12}}{x^2} + {q^{12}}{x^3} + {q^{14}}x + {q^{12}}{x^4}\\
 &+ 2{\mkern 1mu} {q^{14}}{x^2} + {q^{14}}{x^3} + {q^{16}}x + {q^{14}}{x^4} + 2{\mkern 1mu} {q^{16}}{x^2} + 2{\mkern 1mu} {q^{16}}{x^3} + {q^{18}}x + {q^{18}}{x^2} + 2{\mkern 1mu} {q^{18}}{x^3}\\
 &+ {q^{20}}x + {q^{18}}{x^4} + 2{\mkern 1mu} {q^{20}}{x^2} + {q^{20}}{x^3} + {q^{22}}{x^2} + {q^{22}}{x^3} + {q^{22}}{x^4} + {q^{24}}{x^3} + {q^{24}}{x^4} + {q^{26}}x.
 \end{align*}
This proves Conjecture \ref{qCmn-Conjectureb} for $m=5$ and $r=1,2,3,4$.

For the case $n=5k$, we have
$$  H_5^0(x,q) = \PTp_q \frac{q^5-q^{-5}}{q-q^{-1}} F_5(x,q) = \frac{q^{2}\left({1+{q}^{2}{x}^{2}}\right)} {\left({1-{x}^{2}}\right)\left({1-{x}^{3}}\right)\left({1-{q}^{2}x}\right)\left({1-{q}^{4}x}\right)},$$
which clearly implies the positivity for $X_{5,0} H_{5}^0(x,q)$.
Indeed, we have
$$X_{5,0} H_{5}^0(x,q)= \frac{q^{2}P_{5,0}(x,q)} {\left({1-{x}^{2}}\right)\left({1-{x}^{3}}\right)\left({1-{q}^{10}x}\right)\left({1-{q}^{20}x}\right)},
 $$
where $P_{5,0}(x,q) \in \mathbb{N}[q,x]$ contains $64$ terms.
This reproves Conjecture \ref{qCmn-2Conjectureb} for $(m,r)=(5,0)$.
\end{proof}

\subsection{The cases $m\ge 6$}
When we calculated the case $m\ge 6$ in a similar way, we are not able to prove the positivity in a straightforward way as before.
Let us explain the problem by working with the $m=6$ case.
We start with the formula
$$ F_6(x,q)=\frac{q^{12}\left({1-{q}^{2}}\right)\left({1-{q}^{4}}\right)} {\left({x-q}\right)\left({1-qx}\right)\left({x-{q}^{3}}\right)\left({1-{q}^{3}x}\right)\left({x-{q}^{5}}\right)\left({1-{q}^{5}x}\right)\left({1-{q}^{12}}\right)}.$$
By the constant term method, we obtain
\begin{align*}
  G_6(x,q)= \PTp_q  F_6(x,q)=-\frac{xq M_6} {\left({1-qx}\right)\left({1-{q}^{3}x}\right)\left({1-{x}^{6}}\right)\left({1-{q}^{5}x}\right)\left({1-{x}^{8}}\right)\left({1-{x}^{12}}\right)
},
\end{align*}
where $M_6$ is a polynomial of many terms that does not show positivity.
We can apply $X_{6,r}$ for $r=1,5$, corresponding to the $\gcd(m,r)=1$ case. Neither of the two cases shows the positivity directly.

For the $\gcd(m,r)=2$ case, i.e., $r=2,4$, we need to compute
$$ H_6^2(x,q)= \PTp_q \frac{q^2-q^{-2}}{q-q^{-1}} F_6(x,q) = \frac{\text{A lengthy polynomial}}{\left({1-x}\right)^{2}\left({1-{x}^{2}}\right)\left({1-{x}^{4}}\right)\left({1-{q}^{6}x}\right)\left({1-{q}^{18}x}\right)\left({1-{q}^{30}x}\right)
},$$
We can apply $X_{6,r}$ for $r=2,4$. Neither of the two cases shows the positivity directly.

For the $\gcd(m,r)=3$ case, i.e., $r=3$,
we need to compute
$$ H_6^3(x,q)= \PTp_q \frac{q^3-q^{-3}}{q-q^{-1}} F_6(x,q) = \frac{\text{A lengthy polynomial}}{\left({1-x}\right)\left({1-{x}^{2}}\right)\left({1-{x}^{4}}\right)\left({1-{q}^{6}x}\right)\left({1-{q}^{18}x}\right)\left({1-{q}^{30}x}\right)
},$$
We can apply $X_{6,3}$, and the result does not show the positivity directly.

For the case $r=0$, we need to compute
$$ H_6^0(x,q)= \PTp_q \frac{q^6-q^{-6}}{q-q^{-1}} F_6(x,q) = \frac{\text{A lengthy polynomial}}{\left({1-x}\right)^{2}\left({1-{x}^{4}}\right)\left({1-{q}^{6}x}\right)\left({1-{q}^{18}x}\right)\left({1-{q}^{30}x}\right)
},$$
We can apply $X_{6,0}$, and the result does not show the positivity directly.

\subsection{An extension}
The computation of the case $m=6$ suggests that we need to consider different cases for proving Conjecture \ref{qCmn-2Conjectureb}.
However, we find a possible unified way to attack the conjecture. More precisely,
let
\begin{align}\label{Gmxq}
G_m(x,q)=\PTp_q F_m(x,q) \ \text{for} \  m\ge 3,
\end{align}
 or equivalently,
\begin{align}
  [q^i] G_m(x,q)= \CT_q -q^i F_m(x,q) =\CT_q -q^i(q^2-q^{ - 2})\frac{(q-q^{ - 1})}{(q^m-q^{ - m})}\prod_{i = 0}^{m - 1}\frac{1}{ (1-xq^{1 - m}\cdot {q^{2i}})} .
\end{align}
Then $G_m(x,q)\in \mathbb{Q}[q][[x]]$, and it is easy to verify that:
\begin{align*}
 [q^0] G_m(x,q)&=   0,\\
 [qx^n] G_{m}(x,q) &= [q] \N (q^2-1) C_{m,n}(q) \ne 0 \text{ only when $m,n$ are both even},\\
 [q] G_4(x,q) &= -\frac{1}{1-x^4},\\
 [q] G_6(x,q) &= -\frac{x^{6} \left({1-{x}^{2}+{x}^{4}-{x}^{6}-{x}^{8}}\right)} {\left({1-{x}^{6}}\right)\left({1-{x}^{8}}\right)\left({1-{x}^{12}}\right)}\\
   &=-\left({x}^{6}+{x}^{10}+2\,{x}^{18}+{x}^{22}+2\,{x}^{30}+{x}^{34}+2\,{x}^{42}+{x}^{54}\right)+\text{positive terms},\\
 [q] G_{10}(x,q) &= -\left({x}^{6}q+{x}^{10}q\right)+\text{positive terms}.
\end{align*}
We summarized all the other cases in the following conjecture.

\begin{conj}
Let $G_m(x,q)$ be as in \eqref{Gmxq}. Then $G_m(x,q)$ is almost positive for $m\ge 3$. More precisely,
besides the above formula, we have
\begin{enumerate}
\item For every $k\ge 3$, $[q] G_{2k}(x,q)$ has only finitely many negative terms.

\item For every $i\ge 2$, $[q^i] G_m(x,q) \in \mathbb{N}[[x]]$.
\end{enumerate}
\end{conj}

A unified approach for $m=4$ and $r=1,3$ can be given as follows. We have
\begin{align*}
 G_4(x,q)&= -\frac{x q \left({-q+{x}^{3}-{x}^{2}q+{q}^{2}x-{x}^{4}q}\right)} {\left({1-qx}\right)\left({1-{q}^{3}x}\right)\left({1-{x}^{8}}\right)} \\
         &= \frac{q^{2}x \left({1+{x}^{2}-qx-{q}^{2}{x}^{4}+{q}^{3}{x}^{5}}\right)} {\left({1-qx}\right)\left({1-{q}^{3}x}\right)\left({1-{x}^{8}}\right)}-\frac{x^{4}q } {1-{x}^{8}}\\
         &=\frac{q^{2}x \left(1-qx+x^2(1-{q}^{2}{x}^{2})\right)} {\left({1-qx}\right)\left({1-{q}^{3}x}\right)\left({1-{x}^{8}}\right)}+\frac{q^{2}x \left({q}^{3}{x}^{5}\right)} {\left({1-qx}\right)\left({1-{q}^{3}x}\right)\left({1-{x}^{8}}\right)}-\frac{x^{4}q } {1-{x}^{8}}\\
         &=\frac{q^{2}x \left(1+x^2+qx^3\right)} {\left({1-{q}^{3}x}\right)\left({1-{x}^{8}}\right)}  
         +\frac{q^{5}x^6 } {\left({1-qx}\right)\left({1-{q}^{3}x}\right)\left({1-{x}^{8}}\right)}-\frac{x^{4}q } {1-{x}^{8}}
\end{align*}
It follows that
\begin{align*}
  X_{4,r} G_4(x,q) =X_{4,r} \frac{q^{2}x \left(1+x^2+qx^3\right)} {\left({1-{q}^{3}x}\right)\left({1-{x}^{8}}\right)}+X_{4,r} \frac{q^{5}x^6 } {\left({1-qx}\right)\left({1-{q}^{3}x}\right)\left({1-{x}^{8}}\right)},
\end{align*}
which is clearly positive.

%

Generally, for odd $m$, we need to show the positivity of $G_m(x,q)$. It is possible to decompose $G_m(x,q)$ as a sum of trivially positive rational functions.
The decomposition is nontrivial even for $G_5(x,q)$. The even $m$ case needs a minor modification. We succeeded in doing this type of decomposition in \cite{ref-17}, and
hence decomposition of $G_m(x,q)$ for small $m$, at least for $m\le 6$, should be possible. This suggested us to reconsider the following problem in the near future.

\medskip
\noindent
\textbf{Problem:} Given an Elliott rational function $Q$, how to decompose
$Q=\sum_i Q_i$ with $Q_i$ all trivially positive if possible.

\section{Combinatorial model}\label{sec-5-com}

\subsection{Combinatorial interpretation of $C_{m,n}(q)$}
In this section $(m,n)$ is always a coprime pair of positive integers, unless specified otherwise. The general case
$\gcd(m,n)=d>1$ can be considered but is much more complicated.

We should mention that representation theory maybe suitable for settling our conjectures. For instance, Conjecture \ref{qCmn-Conjecture} can be proved if we can find (usually hard to find) an $sl(2)$ module whose character is $q^{-(m-1)(n-1)/2}C_{m,n}(q)$.
This is based on the following well-known result. See, e.g., \cite[Theorem 15]{ref-8}.
\begin{theo}
  Let $\psi: sl(2)\mapsto gl(n)$ be a representation of $sl(2)$ with
  $$\textrm{char } \psi = \sum_{i} b_iq^i.$$
 Then the sequence $\dots, b_{-2},b_{-1},b_0,b_1,b_2,\dots$ is symmetric and unimodal with respect to parity.
\end{theo}

Let $\mathcal{D}_{m,n}$ be the set of Dyck paths in the $m\times n$ lattice rectangle, i.e., paths from $(0,0)$ to $(m,n)$ with unit North step and unit East step, that never go below the diagonal line $y=nx/m$.
The rational $q,t$ Catalan polynomials are defined by
$$ C_{m,n}(q,t)=  \sum_{D \in \mathcal{D}_{m,n}} q^{\area(D)} t^{\dinv(D)},$$
where the sum is over Dyck paths in the $m\times n$ lattice rectangle, $\area(D)$ gives the number of lattice squares between the path and the diagonal,  and $\dinv(D)$ is a Dyck path statistic that can also be given a relatively simple geometric construction. There is also an equivalent interpretation in terms of simultaneous core partitions. See \cite{Armstrong-mn-core,ref-19}.

The rational $q$-Catalan polynomials are specializations of the $q,t$ Catalan polynomials introduced by Garsia and Haiman \cite{Garsia-Haiman}. They have the following combinatorial interpretation:
 $$ C_{m,n}(q) =q^{(m-1)(n-1)/2} C_{m,n}(q,q^{-1}) = \sum_{D \in \mathcal{D}_{m,n}} q^{\coarea(D)+\dinv(D)}.$$ It seems hard to show the parity unimodality of $C_{m,n}(q)$ by this model, because the $\dinv$ statistic is still hard to understand.

A mysterious property of $C_{m,n}(q,t)$ is its symmetry in $q$ and $t$, i.e., $C_{m,n}(q,t)=C_{m,n}(t,q)$. As a symmetric polynomial, $C_{m,n}(q,t)$ has a Schur expansion
$$ C_{m,n}(q,t)= \sum_{\lambda} c_\lambda s_\lambda[q,t],$$
where $\lambda$ can has only two parts, say $\lambda=(\lambda_1,\lambda_2)$, and
$$s_{\lambda_1\lambda_2}[q,t]=(qt)^{\lambda_2}[\lambda_1-\lambda_2]_{q,t}, \qquad \text{where } [k]_{q,t}=q^{k-1}+q^{k-2}t+\cdots +t^{k-1}.$$
Then the $(q,t)$ Schur positivity of $C_{m,n}(q,t)$ (i.e., $c_\lambda\ge 0$ for all $\lambda$) implies the unimodality of $C_{m,n}(q)$ with respect to parity.

The symmetry of $C_{m,n}(q,t)$ is a consequence of the rational shuffle conjecture,
 which can be written as
$$ Q_{m,n}(-1)^n =H_{m,n}(X;q,t),$$
where $H_{m,n}(X;q,t)$ is the Hikita polynomial that has combinatorial interpretation as a sum over rational parking functions \cite{Hikita}, and
$Q_{m,n}$ is a symmetric function operator introduced by Gosky and Negut \cite{Gosky-Negut}.
The rational Shuffle conjecture was proved by Mellit \cite{Mellit}. Detailed definitions are too involved.
The reader is referred to \cite{ref-16} for further information on this topic. We should mention that no combinatorial proof of this symmetry is known up to now.

As a symmetric function in $X$, we can write
$$H_{m,n}(X;q,t)= \sum_{\lambda\vdash n} [s_\lambda]_{m,n} s_\lambda[X].$$
Then $C_{m,n}(q,t)$ is just $[s_{1^n}]_{m,n}$. From the algebraic side, $H_{m,n}(X;q,t)$ is easily seen to be $q,t$ symmetric, so is its coefficients $[s_\lambda]_{m,n}$.
It is then natural to conjecture that $[s_\lambda]_{m,n}(q,t)$ is $(q,t)$ Schur positive. The positivity (though not stated this way)
has been proved for the case $n=2$ by Leven \cite{Leven} and for the case $n=3$ by Qiu and Remmel \cite{ref-16}.


\subsection{Combinatorial interpretation of $C_n(q)$}
Since $C_n(q)=C_{n+1,n}(q)$, we have a combinatorial interpretation of $C_n(q)$. Indeed, let $\mathcal{D}_n$ be short for $\mathcal{D}_{n,n}$. Then we have
$$ C_{n}(q) =\frac{1}{[n+1]_q} \qbinom{2n}{n} = \sum_{D \in \mathcal{D}_{n}} q^{\coarea(D)+\dinv(D)}= \sum_{D \in \mathcal{D}_{n}} q^{\coarea(D)+\bounce(D)}.$$
The second equality follows by the symmetry of $C_{n}(q)$ and by application of the zeta map, which is a bijection from $\mathcal{D}_n$ to itself that takes $\dinv$ to $\area$ and $\area$ to $\bounce$.
See \cite{ref-19,Loehr-higher-qtCatalan}. Usually we think the statistic $\bounce$ is simpler than $\dinv$ in this case, (while for $D\in \mathcal{D}_{m,n}$, the $\dinv(D)$ is known but the $\bounce(D)$ is not).

Currently the simplest way to compute $\coarea(D)$ and $\bounce(D)$ might be as follows (see \cite{vect-k}). Firstly, there is a easy way to convert $D \in \mathcal{D}_{n}$ to a standard Young tableau $T(D)$ of shape $(n,n)$.
Then $\coarea(D)$ is just the sum of the first row entries minus $\binom{n}{2}$, and $\bounce(D)$ is the sum of the first row ranks, where the ranks of the entries of $T$ can be computed in a simple way: i) $r(1)=0$; ii) $r(i)=r(i-1)$ if $i$ is in the first row;
iii) $r(j)=r(i)+1$ if $j$ is under $i$. For example, Figure \ref{fig:bounce} illustrates these statistics for the case $n=3$.

\begin{figure}[!ht]
\centering{
\includegraphics[height=2.3 in]{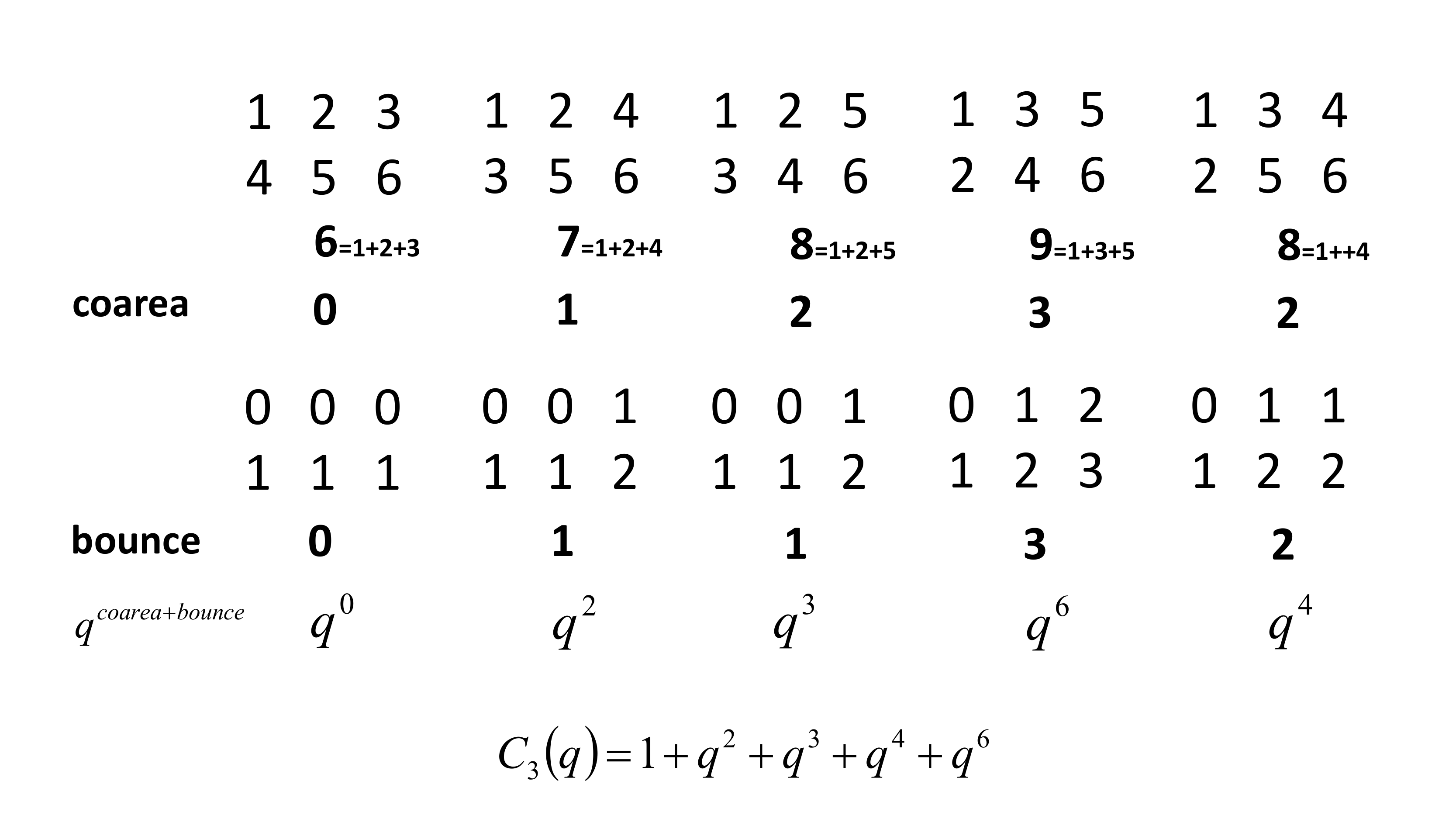}}
\caption{Bounce and coarea for $\mathcal{D}_3$: The top row gives the 5 standard Young tableaux. The bottom row gives the corresponding rank tableaux.\label{fig:bounce}}
\end{figure}

There is a better interpretation found earlier than
the above statistics. See, e.g., \cite{ref-12}.
\begin{align*}
C_n(q)= \frac{1}{[n+1]} \qbinom{2n}{n}&=\sum_{D\in \mathcal{D}_n}q^{maj(D)}
\end{align*}
where $maj(D)$ is the major index of $D$, usually defined as
the sum of the descent positions. (Here a descent corresponds to a $EN$ turn). The major index is also defined for standard Young tableaux. We only
exhibit the major index for $C_3(q)$.

\begin{figure}[!ht]
\centering{
\includegraphics[height=1.1 in]{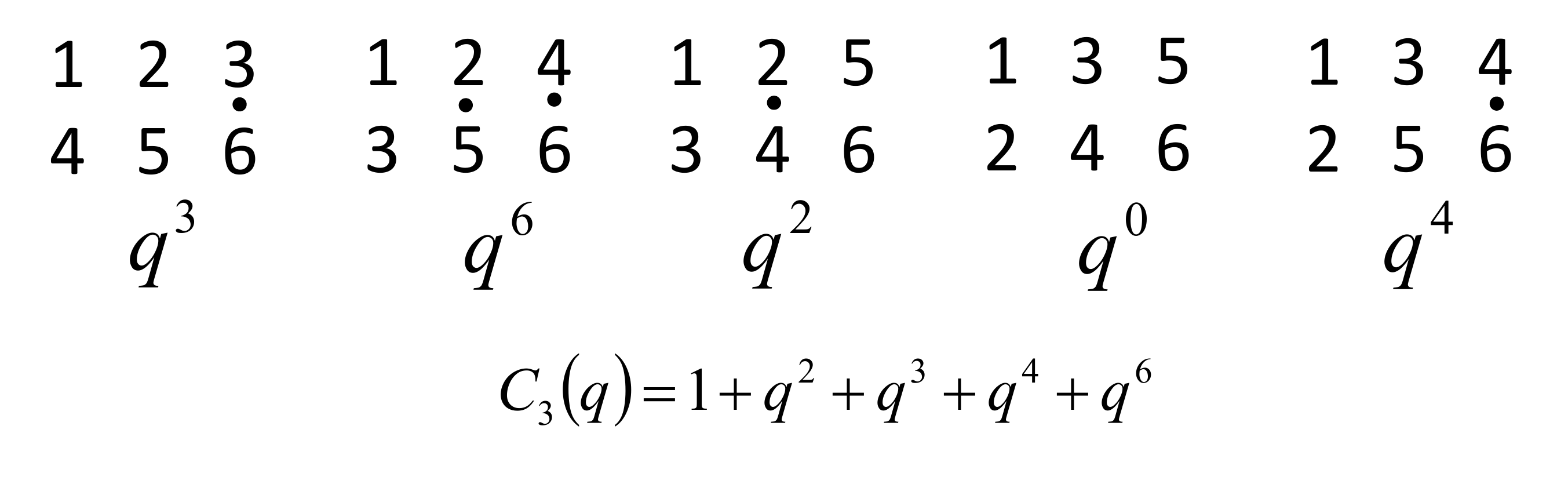}}
\caption{Major index by standard Young tableaux of shape $(n,n)$. Entry $i$ is a descent if $i+1$ appears to the left of $i$.}
\end{figure}

\begin{figure}[!ht]
\centering{
\includegraphics[height=1 in]{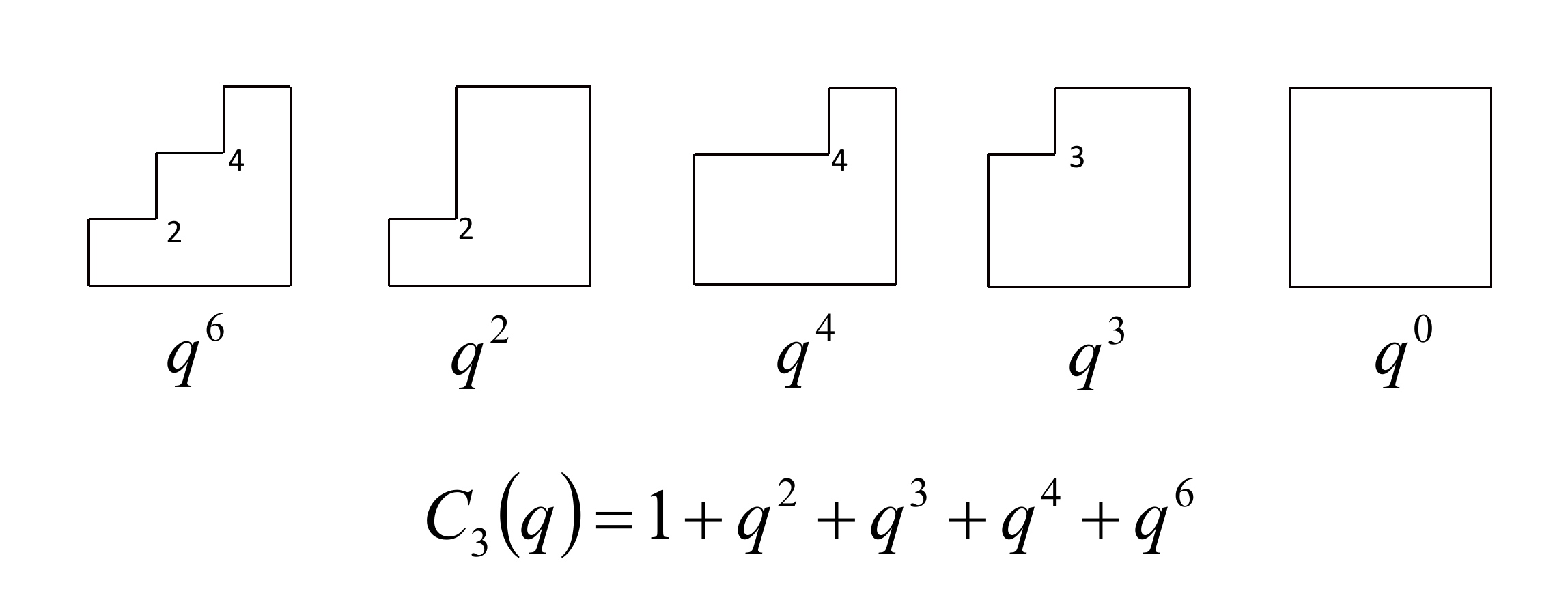}}
\caption{Major index by Dyck paths in $\mathcal{D}_n$. Descents appear at the $EN$ turns.}
\end{figure}

There are also two closely related results. One is the following \cite[p. 523]{ref-8}.
\begin{theo}\label{t-kn-unimodal}
The polynomial $K_n(q)=\frac{1+q}{1+q^n} C_n(q)$  is symmetric and unimodal.
\end{theo}
For instance, $K_0(q)=K_1(q)=1, K_2(q)=1+q, K_3(q)=1+q+q^2+q^3+q^4, K_4(q)=1+q+q^2+2q^3+2q^4+2q^5+2q^6+q^7+q^8+q^9.$ This is also a $q$-analogue of the Catalan number
$C_n= \frac{1}{n+1}\binom{2n}{n}$. The degree of $K_n(q)$ is $(n-1)^2$.

It is not hard to see that the unimodality of $K_n(q)$ implies that $(1+q)C_n(q)=(q^n+1)K_n(q)$ is almost unimodal. Indeed, if we let
$K_n(q)=\sum_{i}k_i q^i$ and $(1+q)C_n(q)= \sum_{i} c_i q^i$, then $c_i=k_i+k_{i+n}$.
Consider
\begin{align*}\label{kn-unimodal}
c_{i+1}-c_i&=(k_{i+1}+k_{i+1+n})-(k_{i}+k_{i+n})=(k_{i+1}-k_{i})+(k_{i+1+n}-k_{i+n}).
\end{align*}
Thus the unimodality of $K_n(q)$ (Theorem \ref{t-kn-unimodal}) implies that $c_{i+1}-c_i\ge 0$ for $0<i<\frac{n^2-4n-1}{2}$, while the desired positivity
$c_{i+1}-c_i\geq 0$ is for $1\le i<\frac{n^2-3n-1}{2}$. 

The other one is the following conjecture. See \cite{ref-13}.
\begin{conj}
Write $C_n(q)=\sum_{k} m_n(k)q^k$. The sequence $(m_n(1),m_n(2),\dots, m_n(n(n-1)-1))$ is unimodal when $n$ is sufficiently large. (Seem to hold for $n\ge 16$.)
\end{conj}
If this conjecture is true for $n\ge 16$, then Conjecture \ref{qCmn-Conjecture} is also true because the $n\le 16$ cases are easily verified to be true.

{\small \textbf{Acknowledgements:}

The authors would like to thank Bill Chen and Arthur Yang for helpful discussions.
%

\bibliographystyle{plain}
\bibliography{bibfile}
\end{document}